\documentclass[11pt,a4paper,twoside]{article}
\usepackage{amsmath,amsfonts,amssymb,amsthm,bbm,latexsym,mathrsfs}
\usepackage{graphicx,color,epsfig,fancyhdr,dsfont,epic,eepic,tikz}
\usepackage{enumerate}
\usepackage{hyperref}
\usepackage{indentfirst}
\usepackage[all]{hypcap}
\usepackage[affil-it]{authblk}
\allowdisplaybreaks
\usepackage[text={15.6cm,21.6cm},top=10mm]{geometry}
\title{On the placement of an obstacle so as to optimize the Dirichlet heat content}

\author[]{Liangpan Li}

	\affil[]{School of Mathematics, Shandong University, Jinan, Shandong, 250100, China\\
liliangpan@gmail.com}	

\date{}

\newtheorem{thm}{Theorem}[section]
\newtheorem{lem}[thm]{Lemma}
\newtheorem{cor}[thm]{Corollary}

\newtheorem{rmk}[thm]{Remark}

\allowdisplaybreaks
\numberwithin{equation}{section}

\begin{document}

\maketitle

\begin{abstract}

We prove that among all doubly connected domains of $\mathbb{R}^n$ ($n\geq2$) bounded by two spheres
of given radii, the Dirichlet heat content at any fixed time achieves
its minimum when the spheres are concentric. This is shown to be a special case of a more general theorem
concerning the optimal placement
of a convex obstacle inside some larger domain so as to maximize or minimize the
Dirichlet heat content.

\medskip

\noindent
{\bf Keywords: Heat kernel, heat content, Dirichlet boundary condition, principle of not feeling the boundary, maximum principles for parabolic equations}
\medskip

\noindent
{\bf Mathematics Subject Classifications (2020):} Primary 35K08; Secondary 49R05
\vskip25pt
\end{abstract}

\pagestyle{fancy}
\fancyhf{}
\fancyhead[LE,RO]{\thepage}
\fancyhead[LO]{\small{Heat content optimization}}
\fancyhead[RE]{\small{L. Li}}

\section{Introduction}

Let $\Omega$ be a bounded path-connected open subset  of  $\mathbb{R}^n$ ($n\geq2$) with smooth boundary,
on which
we consider the (non-negative) Dirichlet Laplacian $\triangle_{\Omega}$ \cite{Davies,Edmunds} with eigenvalues
\[\lambda_1(\Omega)<\lambda_2(\Omega)\leq\lambda_3(\Omega)\leq\cdots\leq\lambda_k:=\lambda_k(\Omega)\leq\cdots\rightarrow\infty\]
and associated normalized  eigenfunctions $\{\phi_k:=\phi_k(\Omega)\}_{k=1}^{\infty}$, that is, $-\triangle \phi_k=\lambda_k\phi_k$ in $\Omega$, $\phi_k|_{\partial\Omega}=0$, and $\int_{\Omega}\phi_k^2=1$, where it is of no harm to assume that $\{\phi_k\}_{k=1}^{\infty}$ are real valued.
It is well known that due to interior regularity \cite[\S 6.3.1]{Evans}
and boundary regularity \cite[\S 6.3.2]{Evans}, $\{\phi_k\}_{k=1}^{\infty}$
are elements of $C^{\infty}(\overline{\Omega})$ \cite[\S 6.5.1]{Evans}.
Recall Weyl's celebrated asymptotic formula \cite{Chavel84,Dodziuk}
\begin{equation}\label{weyl}\lim_{k\rightarrow\infty}\frac{\lambda_k^{n/2}}{k}=\frac{(4\pi)^{n/2}\Gamma(\frac{n+2}{2})}{|\Omega|},\end{equation}
where $|\Omega|$ means the volume of $\Omega$, and  an optimal uniform bound of Grieser \cite{Grieser}
asserting
\begin{equation}\label{grieserlaw}\max_{x\in{\Omega}}|\phi_k(x)|\leq C_{\Omega}k^{\frac{n-1}{2n}}\ \ \ (k\in\mathbb{N}),\end{equation}
where $C_{\Omega}$ is some positive constant depending only on $\Omega$.

The Dirichlet heat kernel $p_{\Omega}(x,y,t)$ for $\Omega$ on $\overline{\Omega}\times\overline{\Omega}\times(0,\infty)$
was originally introduced \cite{Dodziuk,Levi} as
the classical solution  to the heat equation \[\triangle_xp_{\Omega}=\frac{\partial p_{\Omega}}{\partial t}\]
in $\Omega\times(0,\infty)$ subject to
$p(x,y,t)=0$ whenever $x\in\partial\Omega$ and  \[\lim_{t\rightarrow0}\int_{\Omega}p_{\Omega}(x,y,t)u(y)dy=u(x)\]
uniformly for every function $u$
 continuous on $\overline{\Omega}$ and vanishing on $\partial\Omega$.
It is uniquely determined and traditionally written via Mercer's theorem in functional analysis \cite{Lax} as
\begin{equation}\label{heatkernel}p_{\Omega}(x,y,t)=\sum_{k=1}^{\infty}e^{-\lambda_kt}\phi_k(x)\phi_k(y)\ \ \ ((x,y,t)\in\overline{\Omega}\times\overline{\Omega}\times(0,\infty)).\end{equation}
According to (\ref{weyl}) and (\ref{grieserlaw}), we see that the series in (\ref{heatkernel})
converges uniformly on $\overline{\Omega}\times\overline{\Omega}\times[\epsilon,\infty)$ for every $\epsilon>0$.
Actually, by considering (\ref{weyl}), (\ref{grieserlaw}) as well as the elliptic regularity \cite[\S 6.3]{Evans},
one easily gets that $p_{\Omega}$ is a smooth function  satisfying
\[(\partial^{\alpha}p_{\Omega})(x,y,t)=\sum_{k=1}^{\infty}\partial^{\alpha}\big(e^{-\lambda_kt}\phi_k(x)\phi_k(y)\big)\]
for an arbitrary multi-partial derivative $\partial^{\alpha}$ of $2n+1$ variables  \cite{Dodziuk83}.
A striking property of Dirichlet heat kernel is that $p_{\Omega}$
is positive in the interior region $\Omega\times\Omega\times(0,\infty)$ \cite{Dodziuk83,Grigoryan99}. One may also understand  heat kernel
from the viewpoint of probability theory \cite{Simon78,Simon79}, or relate it to wave kernel by considering
the functional calculus
\[e^{-t\triangle_{\Omega}}=\frac{1}{2\sqrt{\pi t}}\int_{\mathbb{R}}\cos(s\sqrt{\triangle_{\Omega}})e^{-\frac{s^2}{4t}}ds\ \ \ (t>0)\]
\cite{Cheeger,Li16,Sikora}.

There are many derived concepts from Dirichlet heat kernel including the Dirichlet heat trace of $\Omega$ defined by
\begin{equation}\label{trace}Z_{\Omega}(t)=\int_{\Omega}p_{\Omega}(x,x,t)dx=\sum_{k=1}^{\infty}e^{-\lambda_kt}\ \ \ (t>0),\end{equation}
and the Dirichlet heat content of $\Omega$ given as
\begin{equation}\label{content}H_{\Omega}(t)=\int_{\Omega}\int_{\Omega}p_{\Omega}(x,y,t)dxdy=\sum_{k=1}^{\infty}e^{-\lambda_kt}\big(\int_{\Omega}\phi_k(x)dx\big)^2 \ \ \ (t>0).\end{equation}
Both functions of positive time are known to have full short-time asymptotic expansions:
\begin{align}
Z_{\Omega}(t) &\sim \frac{|\Omega|}{(4\pi t)^{n/2}}+\sum_{k=1}^{\infty}\alpha_k(\Omega)t^{\frac{k-n}{2}}\ \ \ (t\rightarrow0),\label{FullTrace}\\
H_{\Omega}(t) &\sim |\Omega|+\sum_{k=1}^{\infty}\beta_k(\Omega)t^{\frac{k}{2}}\ \ \ (t\rightarrow0).
\end{align}
A tremendous  amount of effort has been put to establish the existence of both formulae and compute the
coefficients in terms of the geometry of $\partial\Omega$ in the current
\cite{Davies89,BergGilkey94,BergLe94,McKean,MP,Savo98} and more specific
(such as smooth planar regions \cite{Smith}, polygonal domains \cite{BS88,BS90})
or  general (such as vector-valued elliptic operators \cite{Gilkey04,Greiner}) settings.
The Dirichlet spectral zeta function for $\Omega$ is defined as the meromorphic extension
of
\begin{equation}\label{zeta}\zeta_{\Omega}:z\mapsto \sum_{k=1}^{\infty}\frac{1}{\lambda_k^z}=\frac{1}{\Gamma(z)}\int_{0}^{\infty}t^{z-1}Z_{\Omega}(t)dt\ \ \ (\mbox{Re}(z)>\frac{n}{2})\end{equation}
to the complex plane $\mathbb{C}$ whose singularities can be deduced from (\ref{FullTrace})
and basic properties of the Mellin transform \cite{Mellin} to be simple poles at
$\frac{n}{2},  \frac{n-1}{2}, \ldots, \frac{1}{2}$, and  negative half-integers.
The  regularized Dirichlet determinant of $\Omega$, denoted by $\mbox{det}(\Omega)$, is then defined as $\exp(-\frac{d\zeta_{\Omega}}{dz}(0))$.
In much the same way,
\begin{equation}\label{HCspectral}
\mathcal{T}_{\Omega}:z\mapsto \sum_{k=1}^{\infty}\frac{(\int_{\Omega}\phi_k)^2}{\lambda_k^z}=\frac{1}{\Gamma(z)}\int_{0}^{\infty}t^{z-1}H_{\Omega}(t)dt\ \ \ (\mbox{Re}(z)>\frac{n}{2})\end{equation}
admits a  meromorphic extension to  $\mathbb{C}$ whose singularities are
  simple poles at negative half-integers. We also note that
 \begin{equation}k!\cdot\mathcal{T}_{\Omega}(k)=k\int_0^{\infty}t^{k-1}H_{\Omega}(t)dt\ \ \ (k\in\mathbb{N})\end{equation}
  is called $k$-th exit time moment of $\Omega$ \cite{Colloday,Savo}.

In this paper we are particularly interested in  $H_{\Omega}(t)$, which represents the total amount of heat
at time $t$ of the (weak) solution to the heat equation $\triangle_x\psi=\frac{\partial\psi}{\partial t}$
in $\Omega\times(0,\infty)$
with initial temperature 1 and zero boundary conditions on $\partial\Omega\times(0,\infty)$.
Since $\phi_1$ can be assumed to be positive in $\Omega$, we see that as $t$ goes to infinity,
the dominating term of the series in (\ref{content}) is the first one.
In contrast to the heat trace series in (\ref{trace}),  some  higher terms in (\ref{content})
could vanish:
if $\Omega$ is a ball \cite{Folland,Helffer} or an annulus \cite{Li07}, then
\[\int_{\Omega}\phi_2=\cdots=\int_{\Omega}\phi_{n+1}=0,\]
which in turn yields
\[H_{\Omega}(t)=e^{-\lambda_1t}\big(\int_{\Omega}\phi_1\big)^2+\sum_{k=n+2}^{\infty}e^{-\lambda_kt}\big(\int_{\Omega}\phi_k\big)^2 \ \ \ (t>0).\]
A beautiful result of Burchard and  Schmuckenschl\"{a}ger \cite{Burchard} claims that $H_{\Omega}(t)\leq H_{B}(t)$
for all $t>0$ provided that $B$ is an open ball in $\mathbb{R}^n$ having the same volume of $\Omega$.
To compare, Luttinger \cite{Luttinger} established $Z_{\Omega}(t)\leq Z_{B}(t)$ still for all $t>0$,
which, to the best of the author's knowledge, may be the first all-time comparison result  in spectral theory.
An advantage of such kind of  theorems is that they may deliver more information than isoperimetric inequalities
for a single quantity such as the Faber-Krahn inequality  \cite{Chavel84,Chavel01} for the lowest eigenvalue $\lambda_1(\Omega)$.

The main purpose of this paper is to  establish an all-time comparison theorem for the Dirichlet heat content
of domains with ``holes". Our result was motivated by  the paper \cite{ElSoufi2016}
by El Soufi and Harrell concerning a similar result for heat trace.
Let $B$
be an open ball of radius $r_1$ in $\mathbb{R}^n$
such that its closure $\overline{B}$ is contained in another larger concentric open ball $\mathscr{B}$ of radius $r_2$
in $\mathbb{R}^n$. Consider the Dirichlet eigenvalue problem
on $\mathscr{A}_{s}:=\mathscr{B}\backslash(\overline{B}+sV)$,  where $V$ is a fixed unit vector in $\mathbb{R}^n$
and $s\in[0,r_2-r_1)$ is a displacement parameter.

\medskip

  
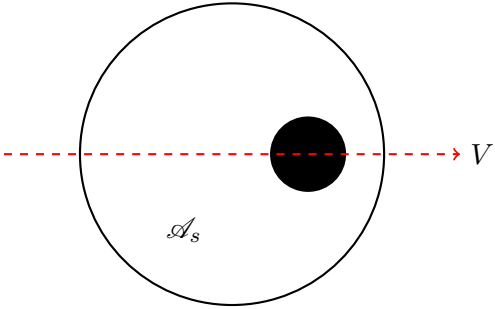
\begin{tikzpicture}




\draw[black,thick] (1,1) circle (2cm);

\fill[black] (2.5,1) arc (0:360:0.5cm) -- cycle;

\draw[red,thick,->,dashed] (-2,1) -- (4,1);

\draw (4,1) node[right] {$V$};

\draw (0,0) node[right] {$\mathscr{A}_{s}$};


\end{tikzpicture}\\
It is known that
\begin{itemize}
  \item (D1)  $\lambda_1(\mathscr{A}_{s})$
is a strictly decreasing function of $s$ \cite{Harrell2001,Kesavan,Ramm} (see also \cite{Hersch}),
\item (D2)  $\lambda_2(\mathscr{A}_{s})$ attains its maximal value
uniquely at $s=0$ \cite{ElsoufiKiwan08},
\item (D3)  $Z_{\mathscr{A}_{s}}(t)$ is a non-decreasing function of $s$ for
every $t>0$ \cite{ElSoufi2016}, \item (D4) $Z_{\mathscr{A}_{s_1}}(t)<Z_{\mathscr{A}_{s_2}}(t)$ as long as $s_1<s_2$ are fixed  and $t>0$ is sufficiently small \cite{Berg04},
\item (D5)   $\mbox{det}(\mathscr{A}_{s})$ is a strictly decreasing function of $s$ \cite{ElSoufi2016}.
\end{itemize}
We mention that (D5) also follows from (D3) and (D4) because due to Kac's principle of not feeling the boundary \cite{Berg81,Berg89,Li16},
the short-time asymptotic coefficients $\{\alpha_k(\mathscr{A}_{s})\}_{k=1}^{\infty}$ in
(\ref{FullTrace}) (with $\Omega$ replaced by $\mathscr{A}_{s}$)
are independent of $s$, from which
it is routine \cite{ElSoufi2016,Strohmaier} to deduce
\[\frac{d\zeta_{\mathscr{A}_{s_2}}}{dz}(0)-\frac{d\zeta_{\mathscr{A}_{s_1}}}{dz}(0)=
\int_{0}^{\infty}\frac{Z_{\mathscr{A}_{s_2}}(t)-Z_{\mathscr{A}_{s_1}}(t)}{t}dt.\]
Apart from the Dirichlet eigenvalue problem  on $\mathscr{A}_s$, one may also
consider Neumann boundary conditions \cite{Anoop21,Verma}, or mixed boundary conditions \cite{Hersch,Payne},
or the Steklov eigenvalue problem \cite{Ftouhi,Paoli},
or the mixed Steklov-Dirichlet problem \cite{Hong,Verma},
or the $p$-Laplacian \cite{Anoop,Chor}, and so on.
Our main result reads as follows.

\begin{thm}\label{thm11}
For any $t>0$, $H_{\mathscr{A}_{s}}(t)$
 is a strictly increasing function of $s$.
\end{thm}

As an immediate corollary, we see that any order exit time moment
of $\mathscr{A}_{s}$ is also a strictly increasing function of the displacement parameter $s$.

Our proof relies on (weak, strong, and Friedman's) maximum principles for parabolic equations \cite{Evans,Friedman,Friedman64,Lieberman,Protter}, Kac's
principle of not feeling the boundary,
and Savo's variational formula
for Dirichlet heat content \cite{Savo}.

\section{Heat kernel comparison}
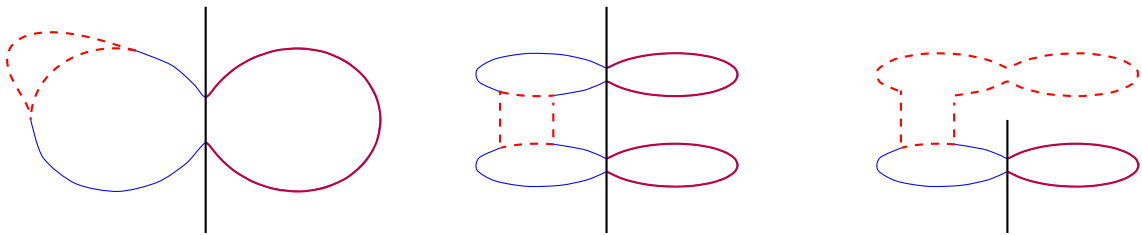
\label{section2}

  \begin{tikzpicture}

\draw[scale=1,domain=-pi:0.75*pi,smooth,color=blue,variable=\t]
plot  ({(1.3+cos(2*\t r))*cos(\t r)},{(1.3+cos(2*\t r))*sin(\t r)});

\draw[scale=1,domain=-pi/2:pi/2,smooth,color=purple,thick,variable=\t]
plot  ({(1.3+cos(2*\t r))*cos(\t r)},{(1.3+cos(2*\t r))*sin(\t r)});

   \draw[scale=1,domain=0.75*pi:pi,smooth,color=red,thick,dashed,variable=\t]
plot  ({(1.3+cos(2*\t r))*cos(\t r)},{(1.3+cos(2*\t r))*sin(\t r)});

\draw[scale=1,domain=0.75*pi:pi,smooth,color=red,thick,dashed,variable=\t]
plot  ({(1.3+cos(2*\t r)+0.7*sin(4*\t r)^2)*cos(\t r)},{(1.3+cos(2*\t r)+0.7*sin(4*\t r)^2)*sin(\t r)});

\draw[black,thick] (0,-1.5) -- (0,1.5);

 \begin{scope}[xshift=150]

  \draw[scale=1,domain=-6*pi/8:9*pi/8,smooth,color=blue,variable=\t]
plot  ({0.75*(1.3+cos(2*\t r))*cos(\t r)},{0.3*(1.3+cos(2*\t r))*sin(\t r)+0.6});

  \draw[scale=1,domain=-pi/2:pi/2,smooth,color=purple,thick,variable=\t]
plot  ({0.75*(1.3+cos(2*\t r))*cos(\t r)},{0.3*(1.3+cos(2*\t r))*sin(\t r)+0.6});

\draw[scale=1,domain=-7*pi/8:-6*pi/8,smooth,color=red,dashed,thick,variable=\t]
plot  ({0.75*(1.3+cos(2*\t r))*cos(\t r)},{0.3*(1.3+cos(2*\t r))*sin(\t r)+0.6});

 \draw[scale=1,domain=7*pi/8:22*pi/8,smooth,color=blue,variable=\t]
plot  ({0.75*(1.3+cos(2*\t r))*cos(\t r)},{0.3*(1.3+cos(2*\t r))*sin(\t r)-0.6});

 \draw[scale=1,domain=-pi/2:pi/2,smooth,color=purple,thick,variable=\t]
plot  ({0.75*(1.3+cos(2*\t r))*cos(\t r)},{0.3*(1.3+cos(2*\t r))*sin(\t r)-0.6});

 \draw[scale=1,domain=6*pi/8:7*pi/8,smooth,color=red,dashed,thick,variable=\t]
plot  ({0.75*(1.3+cos(2*\t r))*cos(\t r)},{0.3*(1.3+cos(2*\t r))*sin(\t r)-0.6});

\draw[red,thick,dashed,] (-0.7,-0.23) -- (-0.7,0.23);

\draw[red,thick,dashed] (-1.4,-0.3) -- (-1.4,0.39);

\draw[black,thick] (0,-1.5) -- (0,1.5);

 \begin{scope}[xshift=150]

\draw[scale=1,domain=-6*pi/8:9*pi/8,smooth,color=red,dashed,thick,variable=\t]
plot  ({0.75*(1.3+cos(2*\t r))*cos(\t r)},{0.3*(1.3+cos(2*\t r))*sin(\t r)+0.6});

\draw[red,thick,dashed,] (-0.7,-0.23) -- (-0.7,0.23);

\draw[red,thick,dashed] (-1.4,-0.3) -- (-1.4,0.39);

\draw[scale=1,domain=6*pi/8:7*pi/8,smooth,color=red,dashed,thick,variable=\t]
plot  ({0.75*(1.3+cos(2*\t r))*cos(\t r)},{0.3*(1.3+cos(2*\t r))*sin(\t r)-0.6});

\draw[scale=1,domain=7*pi/8:22*pi/8,smooth,color=blue,variable=\t]
plot  ({0.75*(1.3+cos(2*\t r))*cos(\t r)},{0.3*(1.3+cos(2*\t r))*sin(\t r)-0.6});

\draw[scale=1,domain=-pi/2:pi/2,smooth,color=purple,thick,variable=\t]
plot  ({0.75*(1.3+cos(2*\t r))*cos(\t r)},{0.3*(1.3+cos(2*\t r))*sin(\t r)-0.6});

\draw[black,thick] (0,-1.5) -- (0,0);

  \end{scope}

 \end{scope}

    \end{tikzpicture}

We  assume that $\Omega$, a  bounded connected open subset of  $\mathbb{R}^n$ with smooth boundary, can be written as the union of  pairwise disjoint non-empty sets \begin{equation}\label{decomposition}\Omega=\Omega_{--}\cup\Omega_{-}\cup (H\cap\Omega\cap\partial\Omega_+)\cup\Omega_+,\end{equation} where
$H$ (in black) is a hyperplane, $\Omega_{+}$ (in purple) and $\Omega_-$ (in blue) are contained in
 distinct connected components of $\mathbb{R}^n\backslash H$ and they are  symmetric with respect to
the hyperplane  $H$, and $\Omega_{--}$ (in red) adheres  somewhere to $\Omega_-$ but
nowhere to $\Omega_+$. In other words, $\Omega\backslash(H\cap\Omega\cap\partial\Omega_+)$ consists of two parts,
one is $\Omega_{+}$ on one side of $\mathbb{R}^n\backslash H$, the other is $\Omega_{--}\cup\Omega_{-}$, and the reflection image of $\Omega_{+}$
with respect to $H$ is a proper subset of $\Omega_{--}\cup\Omega_{-}$.
We remark that $\Omega$ may have several different decompositions of the form (\ref{decomposition}), and each partition is uniquely determined by identifying $\Omega_+$.
For any $x\in\overline{\Omega_+}$, we let $x^{*}$
stand for the reflection point of $x$ with respect to $H$.
As continuous curves in  $\Omega$ connecting two points of $\Omega_+$ may have to pass through
$\Omega_{--}$, we see that $\Omega_+$ is not necessarily connected.
The reflection image of a continuous curve in $\Omega_+$
is  a continuous curve in $\Omega_-$, so $\Omega_{--}\cup\Omega_{-}$ is easily seen to be
 connected.
 To freely apply maximum principles for parabolic equations
 in the space-time $\overline{\Omega_+}\times[0,\infty)$, we further assume  that $\Omega_+$ is  \underline{\emph{connected}},
 unless otherwise stated.
Later on we will see that this condition
can be dropped in many situations (see Remark \ref{rmk27}).

M. van den Berg \cite{Berg81} showed that
\begin{equation}\label{KAC}
\Big|p_{\Omega}(x,y,t)-\frac{1}{(4\pi t)^{n/2}}\exp(-\frac{|x-y|^2}{4t})\Big|\leq\frac{2n}{(4\pi t)^{n/2}}\exp(-\frac{3-2\sqrt{2}}{nt}d_x^2)
\end{equation}
for all $x,y\in\Omega$ and $t>0$, where $d_x:=d(x,\partial\Omega)$ is the distance of $x$ from $\partial\Omega$.
Although this  quantified version of Kac's principle of not feeling the boundary is not optimal for small times \cite{Berg89,Li16,MS},
it well serves the purpose of the current paper.

We will consider  smooth ($C^2$ in $x$ and $C^1$ in $t$ are actually enough) solutions to the heat equation
\begin{equation}\label{heatequation}\triangle_x\psi=\frac{\partial \psi}{\partial t}\end{equation}
in $\Omega_+\times(0,T)$, where $T$ is usually set to be arbitrarily large.
 Suppose some solution  $\psi$
admits a unique continuous extension to
$\overline{\Omega_+}\times[0,T]$, then the weak maximum principle for parabolic equations \cite[p. 368]{Evans}
ensures that both the maximal and minimal values of $\psi$ over the compact cylinder
 $\overline{\Omega_+}\times[0,T]$
are attained on the parabolic boundary $(\partial\Omega_+\times(0,T])\cup (\Omega_+\times\{0\})$ of the cylinder.
This principle does not exclude the possibility of attaining extremal values inside $\Omega_+\times(0,T]$,
and if that situation indeed occurs, say for example at $(x_0,t_0)\in\Omega_+\times(0,T]$, then the strong maximum
principle for parabolic equations \cite[p. 375]{Evans} guarantees that $\psi$ is a constant on $\overline{\Omega_+}\times[0,t_0]$.
We will also apply Friedman's strong maximum principle \cite{Friedman}, which generalizes
Hopf's maximum principle \cite{Gilbarg,Han,Hopf} from elliptic equations to parabolic ones.

\begin{lem}\label{lemma21} For any $x\in H\cap\Omega\cap\partial\Omega_+$, $y\in\overline{\Omega_+}$ and $t>0$, one has
$p_{\Omega}(x,y,t)\leq p_{\Omega}(x,y^*,t)$.
\end{lem}

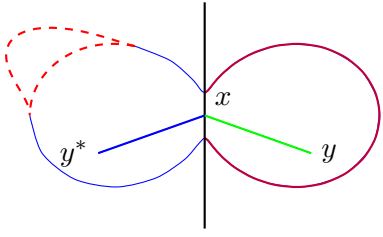
  \begin{tikzpicture}

\draw[scale=1,domain=-pi:0.75*pi,smooth,color=blue,variable=\t]
plot  ({(1.3+cos(2*\t r))*cos(\t r)},{(1.3+cos(2*\t r))*sin(\t r)});

\draw[scale=1,domain=-pi/2:pi/2,smooth,color=purple,thick,variable=\t]
plot  ({(1.3+cos(2*\t r))*cos(\t r)},{(1.3+cos(2*\t r))*sin(\t r)});

   \draw[scale=1,domain=0.75*pi:pi,smooth,color=red,thick,dashed,variable=\t]
plot  ({(1.3+cos(2*\t r))*cos(\t r)},{(1.3+cos(2*\t r))*sin(\t r)});

\draw[scale=1,domain=0.75*pi:pi,smooth,color=red,thick,dashed,variable=\t]
plot  ({(1.3+cos(2*\t r)+0.7*sin(4*\t r)^2)*cos(\t r)},{(1.3+cos(2*\t r)+0.7*sin(4*\t r)^2)*sin(\t r)});

\draw[black,thick] (0,-1.5) -- (0,1.5);


\draw (1.4,-0.5)  node[right] {$y$} -- (0,0)  node[anchor=south west] {$x$} -- (-1.4,-0.5)  node[left] {$y^*$};

\draw[green,thick,] (1.4,-0.5)  -- (0,0);

\draw[blue,thick] (0,0)  -- (-1.4,-0.5) ;


    \end{tikzpicture}

\begin{proof}
Let $x\in H\cap\Omega\cap\partial\Omega_+$ be fixed, and consider
\[\psi(y,t):=p_{\Omega}(x,y,t)-p_{\Omega}(x,y^*,t)\]
on $\overline{\Omega_+}\times(0,\infty)$. It is well known  that $\psi$ is a continuous function (see the Introduction).
Since $x$ is an interior point of $\Omega$, one gets $d_x>0$. Given an arbitrary $y\in\overline{\Omega_+}$,
we now have two cases to consider.

\emph{Case 1}: Suppose $|x-y|<d_x$. Obviously, $|x-y|=|x-y^{*}|$. It then follows from (\ref{KAC}) that
\[|\psi(y,t)|\leq\frac{4n}{(4\pi t)^{n/2}}\exp(-\frac{3-2\sqrt{2}}{nt}d_x^2)\]
for all $t>0$.

\emph{Case 2}: Suppose $|x-y|\geq d_x$. Since  the Dirichlet heat kernel of an arbitrary open domain
is bounded above by the full space counterpart (see e.g. \cite[(3.3)]{Dodziuk}),
one gets
\[|\psi(y,t)|\leq \frac{2}{(4\pi t)^{n/2}}\exp(-\frac{|x-y|^2}{4t}) \leq \frac{2}{(4\pi t)^{n/2}}\exp(-\frac{d_x^2}{4t})\]
for all $t>0$.

Considering both cases, we see that $\psi(y,t)$ converges uniformly to the zero function on $\overline{\Omega_+}$ as $t$ goes to 0.
Consequently, $\psi$ admits a unique continuous zero extension to $\overline{\Omega_+}\times[0,\infty)$.
Next, let  $y\in\partial\Omega_+$ and $t>0$
be arbitrary. If $y\in H\cap\Omega\cap\partial\Omega_+$, then $y=y^{*}$, hence  $\psi(y,t)=0$; else suppose
$y\in \partial\Omega_+\backslash(H\cap\Omega\cap\partial\Omega_+)$, then $p_{\Omega}(x,y,t)=0$,
which implies that $\psi(y,t)\leq0$. To summarize, we see that $\psi$ is non-positive on the parabolic boundary of
$\overline{\Omega_+}\times[0,\infty)$. Obviously, $\psi$ is a smooth solution  (see the Introduction) to the equation
(\ref{heatequation})
in $\Omega_+\times(0,T)$ (the variable $x$ in (\ref{heatequation}) is accordingly changed to $y$) for arbitrarily large $T$,
hence it follows the weak maximum principle for parabolic equations by letting $T\rightarrow\infty$ that $\psi$
is non-positive on $\overline{\Omega_+}\times[0,\infty)$. This proves the lemma.
\end{proof}

\begin{cor}\label{corollary22}  For any $x, y\in\overline{\Omega_+}$ and $t>0$, one has
$p_{\Omega}(x,y,t)\leq p_{\Omega}(x^*,y^*,t)$.
\end{cor}

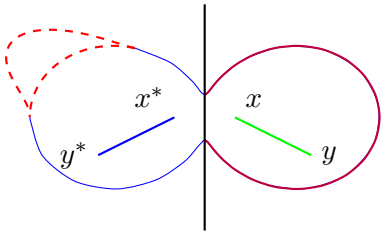
\begin{tikzpicture}

\draw[scale=1,domain=-pi:0.75*pi,smooth,color=blue,variable=\t]
plot  ({(1.3+cos(2*\t r))*cos(\t r)},{(1.3+cos(2*\t r))*sin(\t r)});

\draw[scale=1,domain=-pi/2:pi/2,smooth,color=purple,thick,variable=\t]
plot  ({(1.3+cos(2*\t r))*cos(\t r)},{(1.3+cos(2*\t r))*sin(\t r)});

\draw[scale=1,domain=0.75*pi:pi,smooth,color=red,thick,dashed,variable=\t]
plot  ({(1.3+cos(2*\t r))*cos(\t r)},{(1.3+cos(2*\t r))*sin(\t r)});

\draw[scale=1,domain=0.75*pi:pi,smooth,color=red,thick,dashed,variable=\t]
plot  ({(1.3+cos(2*\t r)+0.7*sin(4*\t r)^2)*cos(\t r)},{(1.3+cos(2*\t r)+0.7*sin(4*\t r)^2)*sin(\t r)});

\draw[black,thick] (0,-1.5) -- (0,1.5);

\draw (1.4,-0.5)  node[right] {$y$} -- (0.4,0)  node[anchor=south west] {$x$};

\draw (-1.4,-0.5)  node[left] {$y^*$} -- (-0.4,0)  node[anchor=south east] {$x^*$};

\draw[green,thick] (1.4,-0.5)  -- (0.4,0);

\draw[blue,thick] (-1.4,-0.5) -- (-0.4,0);

\end{tikzpicture}

\begin{proof} Let $\varphi$ be an arbitrary non-negative function  in $C_c^{\infty}(\Omega_+)$, and let
 $d_{\varphi}$ denote the distance
between the support of $\varphi$, denoted as usual as  $\mbox{supp}(\varphi)$, and  $\partial\Omega_+$. Consider
\[\psi(x,t):=\int_{\Omega_+}(p_{\Omega}(x,y,t)-p_{\Omega}(x^*,y^*,t))\varphi(y)dy\]
on $\overline{\Omega_+}\times(0,\infty)$, which  is easily seen to be a continuous function.
Given an arbitrary $x\in\overline{\Omega_+}$, we  now have two cases to consider.

\emph{Case 1}: Suppose $d(x,\mbox{supp}(\varphi))<\frac{d_{\varphi}}{2}$. This condition implies
$d(x,\partial\Omega_+)\geq\frac{d_{\varphi}}{2}$.
Since any continuous curve with starting point $x$ has to leave $\Omega_+$  first if
it wants to escape from $\Omega$, one gets $d(x,\partial\Omega)\geq d(x,\partial\Omega_+)$. Thus
$d_x\geq\frac{d_{\varphi}}{2}$. Similarly, we have $d(x^{*},\partial\Omega)\geq d(x^*,\partial\Omega_-)$, which
combined with $d(x^*,\partial\Omega_-)=d(x,\partial\Omega_+)$, yields $d_{x^{*}}\geq\frac{d_{\varphi}}{2}$. It then follows from (\ref{KAC})
by considering $|x-y|=|x^*-y^*|$ for all $y\in\Omega_+$ that
\[|\psi(x,t)|\leq\frac{4n}{(4\pi t)^{n/2}}\exp(-\frac{3-2\sqrt{2}}{4nt}d_{\varphi}^2)\int_{\Omega_+}\varphi(y)dy\]
for all $t>0$.

\emph{Case 2}: Suppose $d(x,\mbox{supp}(\varphi))\geq\frac{d_{\varphi}}{2}$.  Then $|x^*-y^{*}|=|x-y|\geq\frac{d_{\varphi}}{2}$ for all
$y\in
\mbox{supp}(\varphi)$. Since  the Dirichlet heat kernel of an arbitrary open domain
is bounded above by the full space counterpart, one gets
\[|\psi(x,t)|\leq\frac{2}{(4\pi t)^{n/2}}\exp(-\frac{d_{\varphi}^2}{16t})\int_{\Omega_+}\varphi(y)dy\]
for all $t>0$.

Considering both cases, we see that $\psi(x,t)$ converges uniformly to the zero function on $\overline{\Omega_+}$ as $t$ goes to 0.
Consequently, $\psi$ admits a unique continuous zero extension to $\overline{\Omega_+}\times[0,\infty)$.
Next, let  $x\in\partial\Omega_+$ and $t>0$
be arbitrary. If $x\in H\cap\Omega\cap\partial\Omega_+$, then it follows from Lemma \ref{lemma21}
that $\psi(x,t)\leq0$; else suppose
$x\in \partial\Omega_+\backslash(H\cap\Omega\cap\partial\Omega_+)$,
then $p_{\Omega}(x,y,t)=0$ for all
$y\in
\mbox{supp}(\varphi)$, which  implies that $\psi(x,t)\leq0$.
 To summarize, we see that $\psi$ is non-positive on the parabolic boundary of
$\overline{\Omega_+}\times[0,\infty)$. Obviously, $\psi$ is a smooth solution to the  equation
(\ref{heatequation})
in $\Omega_+\times(0,T)$ for arbitrarily large $T$,
hence it follows the weak maximum principle for parabolic equations by letting $T\rightarrow\infty$ that $\psi$
is non-positive on $\overline{\Omega_+}\times[0,\infty)$.

Finally, letting $y\in\Omega_+$ be fixed and  $\{\varphi_i\in C_c^{\infty}(\Omega_+)\}_{i=1}^{\infty}$, $\varphi_i\geq0$,  be an approximation
of the Dirac measure at $y$, we get  $p_{\Omega}(x,y,t)\leq p_{\Omega}(x^*,y^*,t)$ for all $x\in\overline{\Omega_+}$ and $t>0$.
This suffices to prove the corollary by continuity.
\end{proof}

\begin{rmk}\label{rmk23} We should remark that Corollary \ref{corollary22} was first given by
 El Soufi and Harrell in  \cite[line 8, p. 890]{ElSoufi2016}.
 They claimed that for \underline{any fixed $x\in\partial\Omega_+$}, one can apply the
 weak maximum principle for parabolic equations \cite[\S 7.1]{Evans} to the solution
 \[\psi(y,t)\mapsto p_{\Omega}(x,y,t)-p_{\Omega}(x^*,y^*,t)\]
of the heat equation $\triangle_y\psi=\frac{\partial\psi}{\partial t}$ in $\Omega_+\times(0,\infty)$.
But $\psi$ does not admit any  continuous extension to  $\overline{\Omega_+}\times[0,\infty)$ if $x$ is an element of $\partial\Omega_+\backslash(H\cap\Omega\cap\partial\Omega_+)$ such that its
reflection $x^*$ with respect to $H$ lies inside $\Omega$ (see the figure at the end of this remark), because due to Kac's principle of not feeling the boundary,
we should have
\[\psi(x,t)=p_{\Omega}(x,x,t)-p_{\Omega}(x^*,x^*,t)=-p_{\Omega}(x^*,x^*,t)\rightarrow-\infty\ \ \ (t\rightarrow0).\]
\end{rmk}

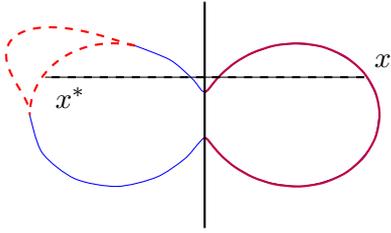
\begin{tikzpicture}

\draw[scale=1,domain=-pi:0.75*pi,smooth,color=blue,variable=\t]
plot  ({(1.3+cos(2*\t r))*cos(\t r)},{(1.3+cos(2*\t r))*sin(\t r)});

\draw[scale=1,domain=-pi/2:pi/2,smooth,color=purple,thick,variable=\t]
plot  ({(1.3+cos(2*\t r))*cos(\t r)},{(1.3+cos(2*\t r))*sin(\t r)});

\draw[scale=1,domain=0.75*pi:pi,smooth,color=red,thick,dashed,variable=\t]
plot  ({(1.3+cos(2*\t r))*cos(\t r)},{(1.3+cos(2*\t r))*sin(\t r)});

\draw[scale=1,domain=0.75*pi:pi,smooth,color=red,thick,dashed,variable=\t]
plot  ({(1.3+cos(2*\t r)+0.7*sin(4*\t r)^2)*cos(\t r)},{(1.3+cos(2*\t r)+0.7*sin(4*\t r)^2)*sin(\t r)});

\draw[black,thick] (0,-1.5) -- (0,1.5);

\draw (2.1,0.5)  node[anchor=south west] {$x$} -- (-2.1,0.5)  node[anchor=north west] {$x^*$};

\draw[black,thick,dashed] (2.1,0.5)  -- (-2.1,0.5);

\end{tikzpicture}

In the rest part of the section we will prepare some strict inequalities for later use.

\begin{thm}\label{thm24} For any $x,y\in \Omega_+$ and $t>0$, one has $p_{\Omega}(x,y,t)<p_{\Omega}(x^*,y^*,t)$.
\end{thm}

\begin{proof}
We argue by contradiction and suppose, by considering Corollary \ref{corollary22},  that there exists $(x_0,y_0,t_0)\in\Omega_+\times\Omega_+\times(0,\infty)$ such that $p_{\Omega}(x_0,y_0,t_0)=p_{\Omega}(x_0^*,y_0^*,t_0)$.
Define
\[\psi(x,t):=p_{\Omega}(x,y_0,t)-p_{\Omega}(x^*,y_0^*,t)\]
on $\overline{\Omega_+}\times(0,\infty)$.
It is straightforward to check that $\psi$ is a smooth solution to the equation
(\ref{heatequation})
in $\Omega_+\times(0,\infty)$, and due to  Corollary \ref{corollary22}, $\psi$ is everywhere non-positive.
Note also
$\psi(x_0,t_0)=0$.
Thus
we can apply,  the strong maximum principle for parabolic equations to the restriction of $\psi$
to $\overline{\Omega_+}\times[\frac{t_0}{2},t_0]$, to see that $\psi$
is  identical to zero on $\overline{\Omega_+}\times[\frac{t_0}{2},t_0]$. In particular,
$\psi(\cdot,t_0)\equiv0$
is equivalent to the fact that
\[p_{\Omega}(x,y_0,t_0)=p_{\Omega}(x^*,y_0^*,t_0)\]
for all $x\in \overline{\Omega_+}$. Recall that the reflection image of $\Omega_{+}$
with respect to $H$ is a proper subset of $\Omega_{--}\cup\Omega_{-}$, so
there exists an element $\widetilde{x}$
of  $\partial\Omega_+\backslash(\Omega\cap H\cap\partial\Omega_+)$
such that its reflection $\widetilde{x}^*$ with respect to $H$ lies inside $\Omega$.
At this point we get
\[p_{\Omega}(\widetilde{x},y_0,t_0)=p_{\Omega}(\widetilde{x}^*,y_0^*,t_0),\]
which is an absurd fact because the left hand side is zero while the right hand side is positive.
This finishes the proof of the theorem.
\end{proof}

\begin{rmk}
Based on the  ``log-concavity" results of Brascamp and Lieb \cite{Lieb}, Ba\~{n}uelos et al. \cite[Prop. 5.2]{banue} showed that for any $t>0$, $p_{\mathbb{U}}(x,x,t)$
is a \underline{strictly decreasing} function of $|x|$, where $\mathbb{U}$ is the  open ball $\{x\in\mathbb{R}^n:|x|<1\}$.
We remark that Theorem \ref{thm24} provides a second proof of this result by suitably identifying  $\Omega$ and $\Omega_+$.
To compare, Pascu and Gageonea \cite{Pascu2,Pascu11}
confirmed a  conjecture of Laugesen and Morpurgo \cite{Laugesen}, which asserts that
the diagonal of the Neumann heat kernel
for  $\mathbb{U}$
is a \underline{strictly increasing} radial function for any fixed time.
\end{rmk}

\begin{thm}\label{thm25} For any $x,y\in\Omega_+$ and $t>0$, one has \[p_{\Omega}(x,y,t)+p_{\Omega}(x,y^*,t)<p_{\Omega}(x^*,y,t)+p_{\Omega}(x^*,y^*,t).\]
\end{thm}

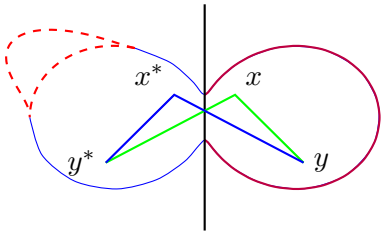
\begin{tikzpicture}

\draw[scale=1,domain=-pi:0.75*pi,smooth,color=blue,variable=\t]
plot  ({(1.3+cos(2*\t r))*cos(\t r)},{(1.3+cos(2*\t r))*sin(\t r)});

\draw[scale=1,domain=-pi/2:pi/2,smooth,color=purple,thick,variable=\t]
plot  ({(1.3+cos(2*\t r))*cos(\t r)},{(1.3+cos(2*\t r))*sin(\t r)});

\draw[scale=1,domain=0.75*pi:pi,smooth,color=red,thick,dashed,variable=\t]
plot  ({(1.3+cos(2*\t r))*cos(\t r)},{(1.3+cos(2*\t r))*sin(\t r)});

\draw[scale=1,domain=0.75*pi:pi,smooth,color=red,thick,dashed,variable=\t]
plot  ({(1.3+cos(2*\t r)+0.7*sin(4*\t r)^2)*cos(\t r)},{(1.3+cos(2*\t r)+0.7*sin(4*\t r)^2)*sin(\t r)});

\draw[black,thick] (0,-1.5) -- (0,1.5);

\draw (1.3,-0.6)  node[right] {$y$} -- (0.4,0.3)  node[anchor=south west] {$x$};

\draw (-1.3,-0.6)  node[left] {$y^*$} -- (-0.4,0.3)  node[anchor=south east] {$x^*$};

\draw[green,thick] (1.3,-0.6)  -- (0.4,0.3) -- (-1.3,-0.6);

\draw[blue,thick] (1.3,-0.6) -- (-0.4,0.3) -- (-1.3,-0.6);

\end{tikzpicture}

As the proof of Theorem \ref{thm25} is so similar to those of Corollary \ref{corollary22}
and Theorem \ref{thm24}, we only present a sketch.

\emph{Step 1}:  Consider the continuous function
\[\psi(x,t):=\int_{\Omega_+}\big(p_{\Omega}(x,y,t)+p_{\Omega}(x,y^*,t)-p_{\Omega}(x^*,y,t)-p_{\Omega}(x^*,y^*,t)\big)\varphi(y)dy\]
on $\overline{\Omega_+}\times(0,\infty)$, where $\varphi\in C_c^{\infty}(\Omega_+)$ is non-negative and fixed.

\emph{Step 2}: Show, by applying (\ref{KAC}) suitably, that $\psi$ admits a unique continuous zero extension to $\overline{\Omega_+}\times[0,\infty)$
that is everywhere non-positive on the parabolic boundary of $\overline{\Omega_+}\times[0,\infty)$.
The weak maximum principle for parabolic equations then ensures that $\psi$ is  non-positive on $\overline{\Omega_+}\times[0,\infty)$.

\emph{Step 3}: Taking a non-negative approximation of the Dirac measure at an arbitrary $y\in\Omega_+$ can yield
 \[p_{\Omega}(x,y,t)+p_{\Omega}(x,y^*,t)\leq p_{\Omega}(x^*,y,t)+p_{\Omega}(x^*,y^*,t)\]
 for all $x\in\overline{\Omega_+}$ and $t>0$.

\emph{Step 4}: Suppose there was $(x_0,y_0,t_0)\in\Omega_+\times\Omega_+\times(0,\infty)$ such that \[p_{\Omega}(x_0,y_0,t_0)+p_{\Omega}(x_0,y_0^*,t_0)=p_{\Omega}(x_0^*,y_0,t_0)+p_{\Omega}(x_0^*,y_0^*,t_0).\]
Then apply the strong maximum principle for parabolic equations to the smooth solution
\[\psi(x,t):=p_{\Omega}(x,y_0,t)+p_{\Omega}(x,y_0^*,t)- p_{\Omega}(x^*,y_0,t)-p_{\Omega}(x^*,y_0^*,t)\]
of the  equation (\ref{heatequation})
in $\Omega_+\times(\frac{t_0}{2},t_0)$ to get $\psi(x,t_0)=0$ for all $x\in\overline{\Omega_+}$.
But by picking an element $\widetilde{x}$
of  $\partial\Omega_+\backslash(\Omega\cap H\cap\partial\Omega_+)$
such that its reflection $\widetilde{x}^*$ with respect to $H$ lies inside $\Omega$, we get
$\psi(\widetilde{x},t_0)<0$, a contradiction.

\begin{thm}\label{thm26} For any $x\in\Omega_+$, $y\in\Omega_{--}\cup\Omega_-$ and $t>0$, one has $p_{\Omega}(x,y,t)<p_{\Omega}(x^*,y,t)$.
\end{thm}

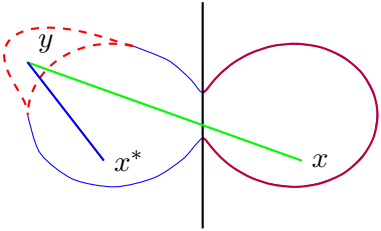
\begin{tikzpicture}

\draw[scale=1,domain=-pi:0.75*pi,smooth,color=blue,variable=\t]
plot  ({(1.3+cos(2*\t r))*cos(\t r)},{(1.3+cos(2*\t r))*sin(\t r)});

\draw[scale=1,domain=-pi/2:pi/2,smooth,color=purple,thick,variable=\t]
plot  ({(1.3+cos(2*\t r))*cos(\t r)},{(1.3+cos(2*\t r))*sin(\t r)});

\draw[scale=1,domain=0.75*pi:pi,smooth,color=red,thick,dashed,variable=\t]
plot  ({(1.3+cos(2*\t r))*cos(\t r)},{(1.3+cos(2*\t r))*sin(\t r)});

\draw[scale=1,domain=0.75*pi:pi,smooth,color=red,thick,dashed,variable=\t]
plot  ({(1.3+cos(2*\t r)+0.7*sin(4*\t r)^2)*cos(\t r)},{(1.3+cos(2*\t r)+0.7*sin(4*\t r)^2)*sin(\t r)});

\draw[black,thick] (0,-1.5) -- (0,1.5);

\draw (1.3,-0.6)  node[right] {$x$} -- (-2.3,0.7)  node[anchor=south west] {$y$} -- (-1.3,-0.6) node[right] {$x^*$};

\draw[green,thick] (1.3,-0.6)  -- (-2.3,0.7);

\draw[blue,thick] (-2.3,0.7) -- (-1.3,-0.6);

\end{tikzpicture}

\begin{proof} Let $y\in\Omega_{--}\cup\Omega_-$ be fixed, and consider the continuous function
\[\psi(x,t):=p_{\Omega}(x,y,t)-p_{\Omega}(x^*,y,t)\]
on $\overline{\Omega_+}\times(0,\infty)$. It is straightforward to check that $\psi$
solves the equation (\ref{heatequation}) in $\Omega_+\times(0,\infty)$ and is non-positive on $\partial\Omega_+\times(0,\infty)$.
Hence for any $0<\epsilon<T<\infty$, the weak maximum principle for parabolic equations ensures that
\[
\max_{(x,t)\in\overline{\Omega_+}\times[\epsilon,T]}\psi(x,t)\leq\max\{0,\sup_{x\in\Omega_+}\psi(x,\epsilon)\}.\]
We then note
\begin{align*}
\sup_{x\in\Omega_+}\psi(x,\epsilon)&\leq\sup_{x\in\Omega_+}p_{\Omega}(x,y,\epsilon)\\
&\leq\sup_{x\in\Omega_+}\frac{1}{(4\pi \epsilon)^{n/2}}\exp(-\frac{|x-y|^2}{4\epsilon})\\
&=
\frac{1}{(4\pi \epsilon)^{n/2}}\exp(-\frac{d(y,\Omega_+)^2}{4\epsilon}),\end{align*}
where it is crucial  to mention  that
\begin{equation}d(y,\Omega_+)>0.\end{equation}
Consequently, by combining the above inequalities and letting first $\epsilon\rightarrow0$ then $T\rightarrow\infty$,
we see that $\psi$ is non-positive
on $\overline{\Omega_+}\times(0,\infty)$. The remaining
 issue of proving $\psi$ being strictly negative in ${\Omega_+}\times(0,\infty)$ is similar to the corresponding part of Theorem \ref{thm24}, thus omitted.
\end{proof}

\begin{rmk}\label{rmk27}
According to the proof of Theorem \ref{thm26}, we point out that the connectedness assumption on $\Omega_+$ in
Lemma \ref{lemma21}, Corollary \ref{corollary22}, Theorems \ref{thm24}, \ref{thm25} and \ref{thm26} all can be dropped as it suffices to
consider the heat equation
(\ref{heatequation}) in the Cartesian product of each individual connected component of $\Omega_+$ with $(0,\infty)$.
Take Corollary \ref{corollary22} for example: in case $x$ and $y$ stay in the same connected component of $\Omega_+$, then
apply  Corollary \ref{corollary22} itself  by redefining $\Omega_+$ as this component (see the top two figures at the end of this remark);
otherwise, apply the approximation technique introduced in the proof of Theorem \ref{thm26} and note $d(y,\Omega_x)>0$, where
$\Omega_x$ denotes the connected component of $\Omega_+$ that contains $x$ (see the bottom two figures).
The other theorems can be dealt with in much the same way.
\end{rmk}

 
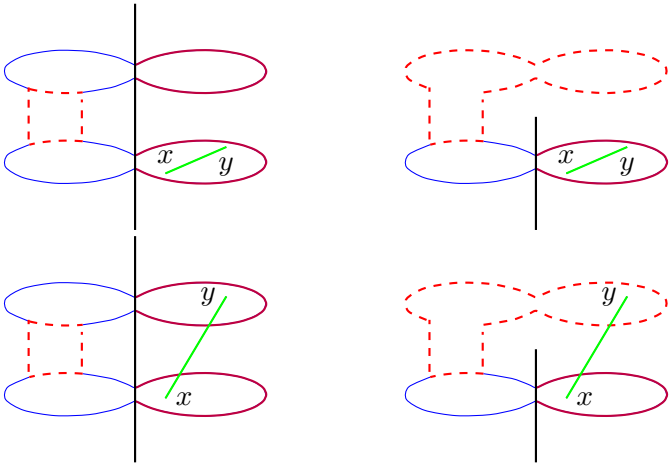
\begin{tikzpicture}

  \draw[scale=1,domain=-6*pi/8:9*pi/8,smooth,color=blue,variable=\t]
plot  ({0.75*(1.3+cos(2*\t r))*cos(\t r)},{0.3*(1.3+cos(2*\t r))*sin(\t r)+0.6});

  \draw[scale=1,domain=-pi/2:pi/2,smooth,color=purple,thick,variable=\t]
plot  ({0.75*(1.3+cos(2*\t r))*cos(\t r)},{0.3*(1.3+cos(2*\t r))*sin(\t r)+0.6});

\draw[scale=1,domain=-7*pi/8:-6*pi/8,smooth,color=red,dashed,thick,variable=\t]
plot  ({0.75*(1.3+cos(2*\t r))*cos(\t r)},{0.3*(1.3+cos(2*\t r))*sin(\t r)+0.6});

 \draw[scale=1,domain=7*pi/8:22*pi/8,smooth,color=blue,variable=\t]
plot  ({0.75*(1.3+cos(2*\t r))*cos(\t r)},{0.3*(1.3+cos(2*\t r))*sin(\t r)-0.6});

 \draw[scale=1,domain=-pi/2:pi/2,smooth,color=purple,thick,variable=\t]
plot  ({0.75*(1.3+cos(2*\t r))*cos(\t r)},{0.3*(1.3+cos(2*\t r))*sin(\t r)-0.6});

 \draw[scale=1,domain=6*pi/8:7*pi/8,smooth,color=red,dashed,thick,variable=\t]
plot  ({0.75*(1.3+cos(2*\t r))*cos(\t r)},{0.3*(1.3+cos(2*\t r))*sin(\t r)-0.6});

\draw[red,thick,dashed,] (-0.7,-0.23) -- (-0.7,0.23);

\draw[red,thick,dashed] (-1.4,-0.3) -- (-1.4,0.39);

\draw[black,thick] (0,-1.5) -- (0,1.5);

\draw (1.2,-0.4)  node[below] {$y$} -- (0.4,-0.75)  node[above] {$x$};

\draw[green,thick] (1.2,-0.4)  -- (0.4,-0.75);

 \begin{scope}[xshift=150]

\draw[scale=1,domain=-6*pi/8:9*pi/8,smooth,color=red,dashed,thick,variable=\t]
plot  ({0.75*(1.3+cos(2*\t r))*cos(\t r)},{0.3*(1.3+cos(2*\t r))*sin(\t r)+0.6});

\draw[red,thick,dashed,] (-0.7,-0.23) -- (-0.7,0.23);

\draw[red,thick,dashed] (-1.4,-0.3) -- (-1.4,0.39);

\draw[scale=1,domain=6*pi/8:7*pi/8,smooth,color=red,dashed,thick,variable=\t]
plot  ({0.75*(1.3+cos(2*\t r))*cos(\t r)},{0.3*(1.3+cos(2*\t r))*sin(\t r)-0.6});

\draw[scale=1,domain=7*pi/8:22*pi/8,smooth,color=blue,variable=\t]
plot  ({0.75*(1.3+cos(2*\t r))*cos(\t r)},{0.3*(1.3+cos(2*\t r))*sin(\t r)-0.6});

\draw[scale=1,domain=-pi/2:pi/2,smooth,color=purple,thick,variable=\t]
plot  ({0.75*(1.3+cos(2*\t r))*cos(\t r)},{0.3*(1.3+cos(2*\t r))*sin(\t r)-0.6});

\draw[black,thick] (0,-1.5) -- (0,0);

\draw (1.2,-0.4)  node[below] {$y$} -- (0.4,-0.75)  node[above] {$x$};

\draw[green,thick] (1.2,-0.4)  -- (0.4,-0.75);

 \end{scope}

    \end{tikzpicture}

     \begin{tikzpicture}

  \draw[scale=1,domain=-6*pi/8:9*pi/8,smooth,color=blue,variable=\t]
plot  ({0.75*(1.3+cos(2*\t r))*cos(\t r)},{0.3*(1.3+cos(2*\t r))*sin(\t r)+0.6});

  \draw[scale=1,domain=-pi/2:pi/2,smooth,color=purple,thick,variable=\t]
plot  ({0.75*(1.3+cos(2*\t r))*cos(\t r)},{0.3*(1.3+cos(2*\t r))*sin(\t r)+0.6});

\draw[scale=1,domain=-7*pi/8:-6*pi/8,smooth,color=red,dashed,thick,variable=\t]
plot  ({0.75*(1.3+cos(2*\t r))*cos(\t r)},{0.3*(1.3+cos(2*\t r))*sin(\t r)+0.6});

 \draw[scale=1,domain=7*pi/8:22*pi/8,smooth,color=blue,variable=\t]
plot  ({0.75*(1.3+cos(2*\t r))*cos(\t r)},{0.3*(1.3+cos(2*\t r))*sin(\t r)-0.6});

 \draw[scale=1,domain=-pi/2:pi/2,smooth,color=purple,thick,variable=\t]
plot  ({0.75*(1.3+cos(2*\t r))*cos(\t r)},{0.3*(1.3+cos(2*\t r))*sin(\t r)-0.6});

 \draw[scale=1,domain=6*pi/8:7*pi/8,smooth,color=red,dashed,thick,variable=\t]
plot  ({0.75*(1.3+cos(2*\t r))*cos(\t r)},{0.3*(1.3+cos(2*\t r))*sin(\t r)-0.6});

\draw[red,thick,dashed,] (-0.7,-0.23) -- (-0.7,0.23);

\draw[red,thick,dashed] (-1.4,-0.3) -- (-1.4,0.39);

\draw[black,thick] (0,-1.5) -- (0,1.5);

\draw (1.2,0.7)  node[left] {$y$} -- (0.4,-0.65)  node[right] {$x$};

\draw[green,thick] (1.2,0.7)  -- (0.4,-0.65);

 \begin{scope}[xshift=150]

\draw[scale=1,domain=-6*pi/8:9*pi/8,smooth,color=red,dashed,thick,variable=\t]
plot  ({0.75*(1.3+cos(2*\t r))*cos(\t r)},{0.3*(1.3+cos(2*\t r))*sin(\t r)+0.6});

\draw[red,thick,dashed,] (-0.7,-0.23) -- (-0.7,0.23);

\draw[red,thick,dashed] (-1.4,-0.3) -- (-1.4,0.39);

\draw[scale=1,domain=6*pi/8:7*pi/8,smooth,color=red,dashed,thick,variable=\t]
plot  ({0.75*(1.3+cos(2*\t r))*cos(\t r)},{0.3*(1.3+cos(2*\t r))*sin(\t r)-0.6});

\draw[scale=1,domain=7*pi/8:22*pi/8,smooth,color=blue,variable=\t]
plot  ({0.75*(1.3+cos(2*\t r))*cos(\t r)},{0.3*(1.3+cos(2*\t r))*sin(\t r)-0.6});

\draw[scale=1,domain=-pi/2:pi/2,smooth,color=purple,thick,variable=\t]
plot  ({0.75*(1.3+cos(2*\t r))*cos(\t r)},{0.3*(1.3+cos(2*\t r))*sin(\t r)-0.6});

\draw[black,thick] (0,-1.5) -- (0,0);

\draw (1.2,0.7)  node[left] {$y$} -- (0.4,-0.65)  node[right] {$x$};

\draw[green,thick] (1.2,0.7)  -- (0.4,-0.65);

 \end{scope}

    \end{tikzpicture}

\begin{rmk}\label{rmk29}
We point out that the smoothness assumption on $\Omega$ in Theorems \ref{thm24}, \ref{thm25} and \ref{thm26}
can be dropped in two steps. Step 1:
Given a  bounded connected open subset  of  $\mathbb{R}^n$  of the form (\ref{decomposition}),
 we can approximate it by an increasing sequence of connected
compactly supported smooth subdomains
to get analogues of these theorems, obtaining inequalities rather than strict inequalities at the moment.
Step 2: These inequalities can then be improved to strict ones by suitably appealing to the strong maximum principle for parabolic equations.
\end{rmk}


\begin{cor}\label{corollary210}
For any $x\in \Omega_+$ and $t>0$, one has
\[\int_{\Omega}p_{\Omega}(x,y,t)dy<\int_{\Omega}p_{\Omega}(x^*,y,t)dy.\]
To be clear, by considering Remarks \ref{rmk27} and \ref{rmk29}, $\Omega\subset\mathbb{R}^n$ is assumed to be a bounded connected open set of the form (\ref{decomposition}).
\end{cor}

\begin{proof}
It follows from Theorem \ref{thm25} (see also Remarks \ref{rmk27} and \ref{rmk29}) that
\[\int_{\Omega_{-}\cup\Omega_+}p_{\Omega}(x,y,t)dy<\int_{\Omega_{-}\cup\Omega_+}p_{\Omega}(x^{*},y,t)dy,\]
and from Theorem \ref{thm26} (see also Remarks \ref{rmk27} and \ref{rmk29}) that
\[\int_{\Omega_{--}}p_{\Omega}(x,y,t)dy<\int_{\Omega_{--}}p_{\Omega}(x^*,y,t)dy.\]
Combining both strict inequalities proves the corollary.
\end{proof}

\section{Heat content optimization}

In  the previous section we have assumed that $\Omega$, a  bounded connected open subset of  $\mathbb{R}^n$ with smooth boundary,
can be written  as (\ref{decomposition})
with $\Omega_+$ being connected.
Now we further require that $\Omega$ is of the form
$\Omega=\Theta\backslash \overline{B}$, where $\Theta$
is some  open subset of $\mathbb{R}^n$, $B$
is a relatively compact open convex subset of $\Theta$
such that it is symmetric with respect to the hyperplane $H$ given in (\ref{decomposition}),
and $\Omega_+$ shares part of its boundary with $B$. Obviously, both $\Theta$ and $B$
are of smooth boundary, and $\Omega_+$ is uniquely determined by $\Omega$, $H$ and $B$.
Let $V$ denote the unit normal vector of $H$ pointing toward $\Omega_+$, and consider
 \[\Omega_{\epsilon}:=\Theta\backslash(\overline{B}+\epsilon V)\]
for $\epsilon\in\mathbb{R}$ with small enough modulus.

\medskip

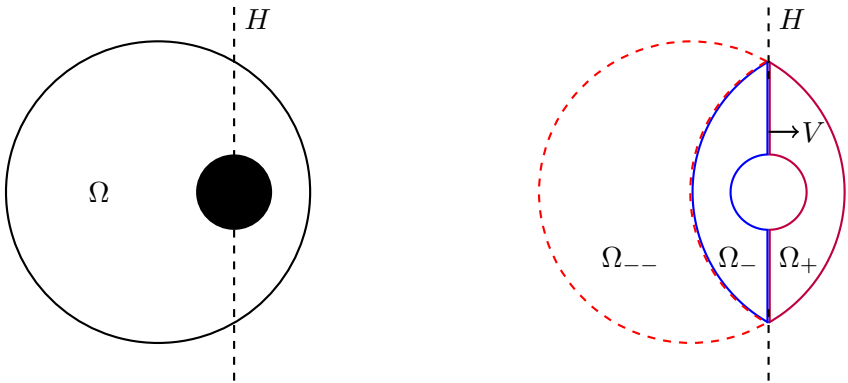
\begin{tikzpicture}


\draw[black,thick] (1,1) circle (2cm);


\draw (0.5,1) node[left] {$\Omega$};

\fill[black] (2.5,1) arc (0:360:0.5cm) -- cycle;

\draw[black,thick,dashed] (2,-1.5) -- (2,3.5);

\draw (2,3.3) node[right] {$H$};

\begin{scope}[xshift=200]

 \draw[scale=1,domain=-pi/3:pi/3,smooth,color=purple,thick,variable=\t]
plot  ({1+2*cos(\t r)},{1+2*sin(\t r)});

 \draw[scale=1,domain=pi/3:5*pi/3,smooth,color=red,thick,dashed,variable=\t]
plot  ({0.98+2*cos(\t r)},{1+2*sin(\t r)});

 \draw[scale=1,domain=-pi/2:pi/2,smooth,color=purple,thick,variable=\t]
plot  ({2+0.5*cos(\t r)},{1+0.5*sin(\t r)});

 \draw[scale=1,domain=2*pi/3:4*pi/3,smooth,color=blue,thick,variable=\t]
plot  ({3+2*cos(\t r)},{1+2*sin(\t r)});

 \draw[scale=1,domain=2*pi/3:4*pi/3,smooth,color=red,thick,dashed,variable=\t]
plot  ({2.97+2*cos(\t r)},{1+2*sin(\t r)});

 \draw[scale=1,domain=pi/2:3*pi/2,smooth,color=blue,thick,variable=\t]
plot  ({2+0.5*cos(\t r)},{1+0.5*sin(\t r)});

\draw[purple,thick] (2.015,1.5) -- (2.015,2.732);

\draw[blue,thick] (1.985,1.5) -- (1.985,2.732);

\draw[purple,thick] (2.015,-0.732) -- (2.015,0.5);

\draw[blue,thick] (1.985,-0.732) -- (1.985,0.5);

\draw (2.8,0.1) node[left] {$\Omega_+$};

\draw (2,0.1) node[left] {$\Omega_-$};

\draw (0.7,0.1) node[left] {$\Omega_{--}$};


\draw[black,thick,dashed] (2,2.5) -- (2,3.5);

\draw[black,thick,dashed] (2,-1.5) -- (2,-0.5);

\draw[black,thick,->] (2,1.8) -- (2.4,1.8);

\draw (2.3,1.8) node[right] {$V$};

\draw (2,3.3) node[right] {$H$};

\end{scope}

\end{tikzpicture}

\begin{thm}\label{thm31} For any $t>0$, one has
\[\frac{dH_{\Omega_{\epsilon}}(t)}{d\epsilon}\Big|_{\epsilon=0}>0.\]
\end{thm}

\begin{proof}
It follows from Savo's variational formula\footnote{We refer the interested reader to \cite{ElSoufi07,Ozawa} for the first variation formula for the Dirichlet heat trace.} \cite[Thm. 10]{Savo} that for any $t>0$,
\[\frac{dH_{\Omega_{\epsilon}}(t)}{d\epsilon}\Big|_{\epsilon=0}=-\int_0^t\Bigg[\int_{\partial B}\langle V,N\rangle\frac{\partial u}{\partial N}(x,\tau)\frac{\partial u}{\partial N}(x,t-\tau)dS(x)\Bigg]d\tau,\]
where $N$ is the unit inner normal to the boundary of $\Omega$, and \[u(x,\tau):=\int_{\Omega}p_{\Omega}(x,y,\tau)dy\ \ \ (x\in\overline{\Omega},\ \tau>0).\]
Note that $(\partial B)\backslash H$ consists of two symmetric connected components,
denoted by $(\partial B)_+$ and $(\partial B)_-$ respectively, and suppose
$(\partial B)_+$ is contained in $\partial\Omega_+$. Thus by symmetry, one gets
\[\frac{dH_{\Omega_{\epsilon}}(t)}{d\epsilon}\Big|_{\epsilon=0}=-\int_0^t\Bigg[\int_{(\partial B)_+}\langle V,N\rangle
\Big[\frac{\partial u}{\partial N}(x,\tau)\frac{\partial u}{\partial N}(x,t-\tau)-\frac{\partial u}{\partial N}(x^*,\tau)\frac{\partial u}{\partial N}(x^*,t-\tau)\Big]dS(x)\Bigg]d\tau.\]
Since $u(x,\tau)$ vanishes on $(\partial B)_+\times(0,\infty)$ and is positive in $\Omega_+\times(0,\infty)$,
we see that
$\frac{\partial u}{\partial N}(x,\tau)\geq0$
for all $x\in(\partial B)_+$ and $\tau>0$. Actually, this trivial inequality can be improved to
\begin{equation}\label{F31}\frac{\partial u}{\partial N}(x,\tau)>0\end{equation}
if we appeal to Friedman's strong  maximum principle concerning
extremal values attained
at \underline{\emph{non-bottom part of the parabolic boundary}}  \cite[Thm. 2]{Friedman} (see also the books \cite{Friedman64,Lieberman,Protter}),
which is an extension of Hopf's  maximum principle from elliptic equations to parabolic ones.
To employ this theorem, it remains to check that $\Omega_+$
 has the interior ball property at every element of $(\partial B)_+$, which obviously stands because of the convexity of $B$.
Applying the same maximum principle to  $u(x^*,\tau)-u(x,\tau)$, which
vanishes on $(\partial B)_+\times(0,\infty)$ and  is positive in $\Omega_+\times(0,\infty)$ because of Corollary \ref{corollary210},
we see that
\begin{equation}\label{F32}\frac{\partial u}{\partial N}(x^*,\tau)>\frac{\partial u}{\partial N}(x,\tau)\end{equation}
for all $x\in(\partial B)_+$ and $\tau>0$.
Observe from the convexity of $B$ that the inner product between $V$ and $N$, as a function of
$x\in(\partial B)_+$, is everywhere  non-negative and assumes  maximal value 1 at points with longest distance to $H$.
Combining this observation with (\ref{F32}), (\ref{F31}), and the variation formula displayed earlier proves the theorem.
\end{proof}

Theorem \ref{thm11} is an immediate consequence of our main result Theorem \ref{thm31}.

We mention  that many illustrating examples about
 fundamental eigenvalue optimization given in \cite{Harrell2001}
 can be transformed into the current heat content context with suitable modifications, notably sign-changing (maximum $\rightleftharpoons$ minimum).

\medskip

\textbf{Acknowledgements}. The author would like to thank Zhirun Zhan for helpful discussions.

\end{document}